\documentclass[reqno]{amsart}
\usepackage{amssymb,latexsym,amscd} 
\usepackage{amsmath}
\makeatletter
\renewcommand{\pod}[1]{\allowbreak\mathchoice
  {\if@display \mkern 18mu\else \mkern 8mu\fi (#1)}
  {\if@display \mkern 18mu\else \mkern 8mu\fi (#1)}
  {\mkern4mu(#1)}
  {\mkern4mu(#1)}
}
\usepackage{enumerate}[1.)]
\usepackage{pb-diagram}  %% needed for the diagrams (commutative diagrams)
\renewcommand{\eqref}[1]{(\ref{#1})}   %for some reason eqref is not supported properly
\theoremstyle{plain}

\newcounter{thm}

\newtheorem{theorem}[thm]{Theorem}

\newtheorem{corollary}[thm]{Corollary}

\newtheorem{lemma}[thm]{Lemma}

\newtheorem{proposition}[thm]{Proposition}
\newtheorem{algorithm}{Algorithm}

\newcommand {\Q}{{\mathbb{Q}}}

\newcommand {\Z}{{\mathbb{Z}}}

\newcommand{\Gal}      {\mathop{\rm {Gal}}}

\newcommand{\Cl}      {\mathop{\rm {Cl}}}

\newcommand{\rk}        {{\mathop{\rm rk}}}

\begin{document}

\title{On Dirichlet biquadratic fields}
\date{\today}
\author{\'Etienne Fouvry}
\address{Universit\' e Paris--Saclay, CNRS, Laboratoire de math\' ematiques d'Orsay, 91405 Orsay, France}
\email{Etienne.Fouvry@universite-paris-saclay.fr}
\author{Peter Koymans}
\address{Max Planck Institute for Mathematics, Vivatsgasse 7, 53111 Bonn, Germany}
\email{koymans@mpim-bonn.mpg.de}
 
%\subjclass[2010]{Primary 11F30; Secondary 11N75}

\begin{abstract}   
We prove the existence of a subset, with positive density, of odd squarefree numbers $n > 0$ such that the $4$--rank of 
the ideal class group of $\Q(\sqrt{-n}, \sqrt{n})$ is $\omega_3(n) - 1$, where $\omega_3(n)$ is the number of prime divisors of $n$ that are $3$ modulo $4$. Recall that for the class groups associated to $\Q(\sqrt{n})$ or $\Q(\sqrt{-n})$ an analogous subset of $n$ does not exist. 
 \end{abstract}

\maketitle

\section{Introduction}\label{intro}
The Cohen--Lenstra  heuristics  \cite{CL} predict the distribution of the $p$-parts of class groups for the family of imaginary quadratic and real quadratic fields, where $p$ is an odd prime. This heuristic has been extended to $p = 2$ by Gerth \cite{Gerth}, and has recently been proven by Smith \cite[Theorem 1.4]{Smith} in a major breakthrough in the context of imaginary quadratic fields. The last paragraph of the introduction of \cite{Ko--Pa} now explains how to develop  Smith's method  in the case of  real quadratic fields.

The Cohen--Lenstra heuristics have been extended to more general families of fields by Cohen and Martinet \cite{CM}. However, their work only deals with the $p$-part with $p$ coprime to the degree of the number field. In this work we explore some of the interesting features when one considers the $2$-part of class groups of biquadratic fields. To the best of our knowledge, there are no heuristics in this setting and it would be an interesting task to develop such heuristics. The odd part of the class group is much better understood, since it is known to be isomorphic to the direct product of the odd parts of the class groups of its quadratic subfields (about this isomorphism, see \cite[p.246]{Lemm} for instance).

To present our results, we will use the following notations
\vskip .3cm
\noindent $\bullet$ if $K$ is a number field, $\mathcal O_K$ is the ring of integers of $K$, $\Cl (K)$ is the ordinary class group of ideals
of $\mathcal O_K$ and $\Cl^+ (K)$ is the  corresponding narrow class group,
\vskip .3cm 
\noindent $\bullet$ 
throughout this paper, $n$ is an odd, squarefree integer  greater than $1$ and $K_n$ is the biquadratic field defined by $K_n := \Q(\sqrt{n},
\sqrt{-n})$. The fields  $K_n$ were first studied by Dirichlet and are called {\it Dirichlet biquadratic fields},
\vskip .3cm
\noindent $\bullet$
if  $G$ is a finite abelian group and $k\geq 1$ is an integer, the $2^k$--rank of $G$ is by definition
$$
\rk_{2^k}  (G):= \dim_{\mathbb F_2} (2^{k-1}G )\Big\slash (2^k G),
$$
 
\vskip .3cm
\noindent $\bullet$
the letters $p$, $p_1$,..., $p'$ are reserved to prime numbers. The total number of prime divisors of $n \geq 1$ is denoted by $\Omega (n)$.  The number of  distinct prime divisors of $n$  is denoted by $\omega (n)$ and the number of  distinct prime divisors congruent to $1$  and $3$ modulo $4$ are  denoted by $\omega_1 (n) $ and $\omega_3 (n)$ respectively,
\vskip .3cm
\noindent $\bullet$ 
we denote by $\mathcal E$ the following {\it exceptional set} 
\[
\mathcal E := \{n > 3 : n\text{ odd, squarefree such that there are } c, e \in \Z \text{ with } c^2 - ne^2 = \pm 2\}.
\]
Considerations on the prime divisors of $n$ and classical sieve techniques imply that $\mathcal E$ has $O (x(\log x)^{-1/2})$ elements     $n\leq x$.
We can now state 
%%%%%%%%%%%%%
\begin{theorem}
\label{central} 
Let $n > 3$ be an odd squarefree  integer satisfying the equalities
\begin{equation}
\label{rk4=rk4=0}
\rk_4 (\Cl ( \mathbb Q (\sqrt n)))=  \rk_4 (\Cl ( \mathbb Q (\sqrt {-n})))=0.
\end{equation}
We then have the equality
$$
\rk_4 (\Cl (K_n)) = \omega_3 (n) + \delta(n) + \epsilon(n) - 1,
$$
where $\epsilon(n) = 1$ if $n \in \mathcal{E}$ and $0$ otherwise, and $\delta(n)$ is defined as follows
$$
\delta(n) =
\begin{cases}
1 & \textup{ if } \omega_3 (n) =0 \textup{ and } \exists\, p \mid n,\, p \equiv 5 \bmod 8,\\
0 & \textup{ if } \omega_3 (n) =0 \textup{ and }   p \mid  n\Rightarrow  p \not\equiv 5 \bmod 8,\\
0 &  \textup{ if } \omega_3 (n) \geq 1 \textup{ and } \exists\, p \mid n,\, p \equiv 5 \bmod 8,\\
-1 & \textup{ if } \omega_3 (n) \geq 1 \textup{ and }   p \mid n \Rightarrow  p \not\equiv 5 \bmod 8.\\
\end{cases}
$$
\end{theorem}

In particular it follows that with at most $O(x/(\log x)^{1/4})$ exceptions, we have that $\rk_4 (\Cl (K_n)) = \omega_3(n) - 1$ for the set of $n < x$ satisfying equation \eqref{rk4=rk4=0}.

\begin{corollary} 
\label{<rk4<}
For every odd squarefree integer $n > 3$ satisfying \eqref{rk4=rk4=0}
we have the inequalities
$$
\max( \omega_3 (n) -2, 0) \leq \rk_4 (\Cl( K_n)) \leq  \omega_3 (n) +1.
$$
\end{corollary}

Corollary \ref{<rk4<} gives a rather precise formula for the function $\rk_4 (\Cl (K_n))$ provided that 
$n$ satisfies \eqref{rk4=rk4=0}. Actually, the set of such $n$ is rather dense since it  represents about $28 \%$ of the set of odd squarefree 
integers (see Proposition \ref{eta(2)} below). Furthermore, a classical result of analytic number theory asserts that most of the integers $n$
are such that $\omega (n)$ is very close to its average value $\log \log n$. The same concentration phenomenom holds for the function $\omega_3 (n)$
around the value $(1/2) \log \log n$ (see Proposition \ref{concentration} below). Combining these facts with Corollary \ref{<rk4<}, we deduce the following.

\begin{corollary}
Let $\varepsilon >0$ be given. Then, for  $x>x_0 (\varepsilon)$ we have that
$$
\frac{\vert \{3 \leq n \leq x : n \textup{ odd and squarefree, } \vert \rk_4 (\Cl(K_n)) -(1/2) \log \log n\vert \leq \varepsilon \log \log n\} \vert }{\vert \{3 \leq n \leq x : n\textup{ odd and squarefree}\} \vert}
$$
is at least $0.28$.
\end{corollary}

This corollary shows that a positive proportion of odd squarefree $n$ are such that the associated Dirichlet biquadratic field has a large $4$--rank. Since the function $\log \log n$ tends to infinity with $n$, we see that the situation is completly different for the case of the quadratic fields $\mathbb Q (\sqrt n)$ and $\mathbb Q (\sqrt {-n})$, where the $4$--rank has an average  tendency to be much smaller. For a precise statement see  Proposition \ref{nodensity} below. 

%%%%%%%%%%%%%%%%%
\section{Notations,  lemmas and propositions}\label{sec2}
%%%%%%%%%%%%%%%%%
\subsection{Notations} We complete the notations given in the introduction by the following
\vskip .3cm
\noindent $\bullet$
to shorten notations, we write $\Cl(n)$ and $\Cl^+ (n)$ for   $\Cl(\Q(\sqrt{n}))$ and  $\Cl^+(\Q(\sqrt{n}))$ respectively. Their cardinalities
are denoted by $h(n)$ and $h^+ (n)$, 
\vskip .3cm
\noindent $\bullet$ in the ring  of integers $\mathcal O_K$, the subset of invertible elements is denoted by $\mathcal O_K^*$,
\vskip .3cm
\noindent $\bullet$
we also introduce the following  set
$$
\mathcal D := \{ D :  D\text{ fundamental discriminant}\},
$$
and its two subsets $\mathcal D^+$ and $\mathcal D^-$ containing the fundamental $D>0$ and $D<0$ respectively, 
and finally
\[
\mathcal N:= \{n \geq 3 \text{ odd, squarefree} :\rk_4({\Cl}^+(n)) = \rk_4 (\Cl(-n)) = 0\},
\]
\vskip .3cm
\noindent$\bullet$ for every integer $r\geq 0$, the function $\eta_r : \mathbb R^+ \rightarrow \mathbb R^+$ is defined by the formula
$$
\eta_r (t) := \prod_{j=1}^r (1-t^{-j}).
$$
%%%%%%%%%%%%%%
\subsection{The  $2$--rank of class groups in the case of quadratic fields} 
The first lemma is famous and was proved by Gauss in the context of binary quadratic forms.

\begin{lemma} 
\label{Gauss1}
Let $D$ be an element of $\mathcal D$.  Then we have
$$
\rk_2 ({\Cl}^+ ( D)) = \omega (\vert D\vert) -1.
$$
\end{lemma}

If $D$ is negative, the two class groups ${\Cl}^+(D)$ and $\Cl (D)$ coincide. This may be not the case when $D>0$. However we have

\begin{lemma}
\label{Gauss2}
Let $D$ be an element of  $\mathcal D^+$. We have
\begin{enumerate} 
\item [i)]$\Cl (D)$ is a factor group of $\Cl^+ (D)$ with index $\mathfrak i (D) \in \{1, 2\}$, in particular we have the inequalities
$$ \rk_2({\Cl}^+ (D)  )-1\leq  \rk_2(\Cl (D)) \leq  \rk_2({\Cl}^+ (D)),$$ 
\item [ii)] $\mathfrak i (D)$ is equal to $2$ if and only if, the fundamental unit of $\mathcal O^*_{\mathbb Q(\sqrt D)} $ has norm~$1$,\\
\item [iii)] suppose that $D$ is such that  $\mathfrak i (D) =2$, then we have  the equality
$$
\rk_2 ({\Cl}(D))= \rk_2 ({\Cl}^+ (D))-1
$$
if and only if $D$ is divisible by some $p \equiv 3 \bmod 4$. 
\end{enumerate}
\end{lemma}

For comments on this last item, see \cite[Lemma 1]{Fo--Kl4}. From Lemmas \ref{Gauss1} and \ref{Gauss2} we deduce the following statement where we restrict to odd $n$.
 
\begin{lemma}
\label{Gauss3} Let $n\geq 3$ be an odd squarefree integer. Then we have
\begin{equation}
\label{rk2=}
\rk_2 (\Cl (n)) = \begin{cases}
\omega (n)-2 & \text{ if } n \equiv 1 \bmod 4\, \text{ and }\, \omega_3 (n)\geq 1,\\
\omega (n)-1 & \text{ if } n \equiv 1 \bmod 4\, \text{ and } \omega_3 (n)=0,\\
\omega (n)-1 & \text{ if } n\equiv 3 \bmod 4,
\end{cases}
\end{equation}
and also 
\begin{equation}
	\label{rk2==}
\rk_2 (\Cl (-n))=
\begin{cases}
\omega (n) -1 &\text{ if } n\equiv 3 \bmod 4,\\
\omega (n) & \text{ if } n\equiv 1 \bmod 4.
\end{cases}
\end{equation}
\end{lemma}

\begin{proof} 
It is based on the fact that the discriminant of the field $\mathbb Q (\sqrt {\pm n})$ is either $\pm n$ or $\pm 4n$  according to the congruence class of $\pm n\bmod 4$.
\end{proof}

We gather \eqref{rk2=} and \eqref{rk2==} of Lemma \ref{Gauss3} in the following statement which will be useful in \S \ref{Theproof}.

\begin{proposition}
\label{Gauss4}
Let $n\geq 3$ be an odd squarefree integer. We then have the equality
$$
\rk_2 (\Cl (n)) + \rk_2 (\Cl (-n))=
\begin{cases} 
2 \omega (n) -1 &\text{ if } \omega_3 (n)=0,\\
2 \omega (n) -2 & \text{ if } \omega_3 (n) \geq 1.
\end{cases}
$$
\end{proposition}
%%%%%%%%%%%%%%%
\subsection{The  $2$--rank of class groups in the case of Dirichlet biquadratic fields}
We will prove the following 

\begin{proposition}
\label{rk2Kn=} 
Let $n \geq 3$ be a squarefree odd integer. We have the equality
$$
\rk_2 (\Cl (K_n) )=
\begin{cases} 
 {2 \omega_1 (n) +\omega_3 (n) -1} & \text{ if }p \mid n \Rightarrow p \not\equiv 5 \bmod 8,\\
 {2 \omega_1 (n) +\omega_3 (n)-2  } & \text{ if } \exists\,  p \mid n \text{ and } p\equiv 5 \bmod 8.
\end{cases}
$$ 
\end{proposition}

The following lemma plays a central role.

\begin{lemma}
\label{rk2(Kn)=} 
For all odd squarefree integers $n \geq 3 $  we have the equality
\begin{equation}
\label{265}
2^{\rk_2 (\Cl (K_n))} = \frac{1}{4} |\{\beta \in \Z[i] : \beta \mid n, \beta \equiv \pm 1 \bmod 4\}|.
\end{equation}
\end{lemma}

\begin{proof}
Since $\Q(i)$ is a PID, we see that the action of $\Gal(K_n/\Q(i))$ on $\Cl(K_n)$ is by inversion. In particular, we see that any unramified, abelian extension $L$ of $K_n$ is Galois over $\Q(i)$. Furthermore, the exact sequence
\[
1 \rightarrow \Gal(L/K_n) \rightarrow \Gal(L/\Q(i)) \rightarrow \Gal(K_n/\Q(i)) \rightarrow 1
\]
is split. Indeed, using once more that $\Q(i)$ is a PID, we see that $K_n/\Q(i)$ is ramified at some place, so a splitting is given by inertia. Hence any quadratic unramified extension $L$ of $K_n$ has Galois group $C_2 \times C_2$ over $\Q(i)$ and is therefore of the shape $K_n(\sqrt{\beta})$ with $\beta \in \Z[i]$. 

By straightforward ramification considerations we see that $K_n(\sqrt{\beta})/K_n$ is unramified at all odd places if and only if $\beta \mid n$. Now a local computation at $\Q_2(i)$, which is the completion of $K_n$ at any prime above $2$, shows that $K_n(\sqrt{\beta})/K_n$ is unramified at any prime above $2$ if and only if $\beta \equiv \pm 1 \bmod 4$. 

For  $\alpha \in K_n^\ast$, let $\chi_\alpha$ is the continuous group homomorphism $G_{K_n} \rightarrow \{\pm 1\} $ given by $\sigma \mapsto \chi_\alpha (\sigma) =\bigl( \frac{\sigma(\sqrt{\alpha})}{\sqrt{\alpha}}\bigr)$.
We have now shown that  the group of characters $\Cl(K_n)^\vee[2]$ is generated by the characters $\chi_\beta$ with $\beta \in \Z[i]$, $\beta \mid n$ and $\beta \equiv \pm 1 \bmod 4$. But the characters $\chi_{-1}$ and $\chi_n$ are obviously trivial when restricted to $K_n$. The remaining characters are linearly independent and the result follows.
\end{proof}

It remains to count the cardinality, denoted by $F(n)$, of the set  of divisors $\beta$ appearing in the right--hand side of \eqref{265}. We have

\begin{lemma}
\label{countingdivisors} 
Let $n\geq 3$ be an odd squarefree integer. We then have the equality 
$$
F(n)=
\begin{cases} 
2^{2 \omega_1 (n) +\omega_3 (n) +1} & \text{ if }p \mid n \Rightarrow p \not\equiv 5 \bmod 8,\\
 2^{2 \omega_1 (n) +\omega_3 (n)  } & \text{ if } \exists \, p \mid n \text{ and } p\equiv 5 \bmod 8.
\end{cases}
$$
\end{lemma}

\begin{proof}
Let $\phi_n : \{\beta \in \Z[i] : \beta \mid n\} \rightarrow (\Z[i]/4\Z[i])^\ast$ be the map of $\mathbb{F}_2$-vector spaces given by $z \mapsto z \mod 4\Z[i]$. Then $F(n) = |\phi_n^{-1}(\{1, -1\})|$. Observe that
\[
|\{\beta \in \Z[i] : \beta \mid n\}| = 2^{2\omega_1(n) + \omega_3(n) + 2}.
\]
Now the lemma follows immediately if we can prove that
\[
\text{im}(\phi_n) = 
\begin{cases} 
\{1, -1, i, -i\} & \text{ if }p \mid n \Rightarrow p \not\equiv 5 \bmod 8,\\
(\Z[i]/4\Z[i])^\ast & \text{ if } \exists \, p \mid n \text{ and } p\equiv 5 \bmod 8.
\end{cases}
\] 
Clearly, $\{1, -1, i, -i\} \subseteq \text{im}(\phi_n)$. Furthermore, simple considerations of congruences modulo $8$ imply that 
$$
\phi_n(\beta) \not \in \{1, -1, i, -i\} \Leftrightarrow N_{\mathbb Q(i)/\mathbb Q} \, (\beta) \equiv 5 \bmod 8.
$$
This concludes the proof.
\end{proof}
%%%%%%%%%%%%%%%%%%%%%
\subsection{The $4$--rank of class groups in the case of quadratic fields}  
We recall a weaker form of the results of Fouvry and Kl\" uners (\cite{Fo--Kl1, Fo--Kl2}) giving strong evidence for the truth of the Cohen--Lenstra--Gerth heuristics (see \S \ref{intro}).

\begin{lemma}
\label{4rankFK}
For every integer $r\geq 0$, there exists two constants $\alpha^+_r >0$ and $\alpha^-_r >0$, satisfying the equalities
$$
\sum_{r\geq 0} \alpha_r^\pm =1,
$$
and such that,  as $X \rightarrow +\infty$, one has
$$
\vert \{ D \in \mathcal D^\pm : 0< \pm D < X, \, \rk_4 (\Cl (D)) =r \} \vert \sim \alpha_r^{\pm} \vert \{  D \in \mathcal D^\pm : 0< \pm D < X\} \vert.
$$
Similar statements hold with the following three congruence restrictions on the fundamental $D$ :
$$
D \equiv 1 \bmod 4, \, D \equiv 4 \bmod 8,  \, D \equiv 0 \bmod 8.
$$
\end{lemma}

The constants $\alpha_r^\pm$ do not depend on the three congruences  above  and they can be expressed in terms of the $\eta_r$--function. From the remark that a squarefree odd $n\geq 3$ is either a fundamental discriminant (when $n \equiv 1 \bmod 4$) or one quarter of a fundamental discriminant (when $n \equiv 3 \bmod 4$) we deduce from Lemma \ref{4rankFK} the following

\begin{proposition}
\label{nodensity}
Let $\psi : \mathbb R^+ \rightarrow \mathbb R^+$ be an increasing function tending to infinity. Then as $X$ tends to infinity, we have
$$
\vert \{ n : 1\leq n\leq X, \, n \textup{ odd and squarefree, } \rk_4 (\Cl (n) )\geq \psi (n)\}\vert  = o ( X).
$$
A similar statement holds for negative $n$.
\end{proposition}

We will appeal to another result of Fouvry and Kl\" uners concerning the values of the pair $(\rk_4 (\Cl (n)), \rk_4 (\Cl (-n)))$. 
We have

\begin{lemma} 
\cite[Theorem 1.8]{Fo--Kl3}
For any integer  $r\geq 0$ we have
\begin{multline*}
\lim_{X \rightarrow \infty}\frac {\vert\{D \in \mathcal D^+ : D \equiv 1 \bmod 4,D \leq X,  \rk_4({\Cl}^+(D)) = \rk_4 (\Cl(-D)) = r\}\vert }{\vert \{ D \in \mathcal D^+ : D \equiv 1 \bmod 4,D\leq X\} \vert} \\ = 2^{-r^2-r}(1-2^{-(r+1)})\eta_\infty (2) \eta_r(2)^{-1} \eta_{r+1}(2)^{-1}.
\end{multline*}
A similar statement holds true when the congruence $D \equiv 1 \bmod 4$ is replaced by $D\equiv 4 \bmod 8$.
\end{lemma}

Apply this lemma with $r=0$. Since  $n$ or $4n$ is  a fundamental discriminant and computing numerically the value of $\eta_\infty (2)$, we deduce

\begin{proposition}
\label{eta(2)}
We have 
$$
\lim_{X \rightarrow \infty}\frac { \vert\{ 1 \leq n \leq X : n \in \mathcal N   \}\vert }{\vert \{1 \leq n\leq X : n\textup{ odd, squarefree} \} \vert} \\ =  \eta_\infty (2)
=\,.28878\, 80950\,\dots 
$$
\end{proposition}
%%%%%%%%%%%%%%
\subsection{The $8$--rank of some Dirichlet biquadratic fields.}  
\begin{proposition}  
\label{8rank}
For all $n \in \mathcal N$ we have 
$$
\rk_8 ( \Cl (K_n) ) = 0.
$$
\end{proposition}

\begin{proof}
This is a particular case of  Lemma 3.2 of \cite{Chan--Milovic}.
\end{proof}
%%%%%%%%%%%%%
\subsection{A formula due to Dirichlet}
\begin{proposition}
\label{hh=h} Let $n\geq 3$ be an odd squarefree integer. Then we have the equality
\begin{equation} 
\label{formdir}
\frac 12 h(n) h(-n) Q(n) = \bigl\vert \Cl (K_n)\bigr\vert,
\end{equation}
where $Q(n)$ takes the values $1$ or $2$. More precisely $Q(n)$ is the Hasse unit index
$$
Q(n) := \left[\mathcal{O}_{K_n}^\ast : \mathcal{O}_{\Q(\sqrt{n})}^\ast \Z[i]^\ast\right].
$$
\end{proposition}

This formula, in the context of quadratic forms,  was first discovered by Dirichlet, who derived it from his class number formula, see \cite{Dir}. It can for instance be found in the famous Zahlbericht of Hilbert \cite[Theorem 115 p.153]{Hilb}.
 
 %%%%%%%%
\subsection{Study of the Hasse unit index}
\begin{proposition} 
\label{Hasse}  
Let $n > 3$ be an odd squarefree integer. Then we have $Q(n) = 2$ if and only if $n \in \mathcal{E}$.
\end{proposition} 

\begin{proof}
We proceed in two steps. We claim that 
\begin{equation}
\label{iff1}
Q(n) = 2 \iff \textit{there exists  }\epsilon \in \mathcal{O}_{\Q(\sqrt{n})}^\ast\textit{  and } z \in K_n \textit{ such that }
z^2 = \epsilon i.
\end{equation}
Let $A$ and $B$ the multiplicative groups $A:= \mathcal{O}_{\Q(\sqrt{n})}^\ast \Z[i]^\ast $ and $B:= \mathcal{O}_{K_n}^\ast$. By definition $A$ is a subgroup of $B$ and we know that its index $Q(n)$ is equal to $1$ or $2$. The roots of unity in $K_n$ are $\pm 1$ and $\pm i$ thanks to our assumption $n > 3$. Then by the Dirichlet Unit Theorem, $A$ and $B$ are isomorphic to
\begin{equation}
\label{iso}
A, B \cong (\mathbb Z/4 \mathbb Z) \oplus \mathbb Z.
\end{equation} 
Suppose that $Q(n)=2$. Then we get a short exact sequence of abelian groups
$$ 
1 \longrightarrow  A \overset{\imath}{\longrightarrow} B \overset{\pi}{\longrightarrow}\{ \pm 1\}\longrightarrow 1, 
$$ 
where $\imath$ is the natural injection, and $\pi$ the canonical projection on $B/A$. Let $b\in B$ such that $\pi (b) =-1$.
Then $b$ is not in the image of $A$. However $\pi (b^2) =1$, so $b^2$ is in $A$. We claim that $b^2$ is not of the shape
$b^2= a^2$ with $a\in  A$. If it was the case there would exist a morphism $\psi$  from $\{\pm 1\}$ to $B$ such that $\psi(-1) =a/b$. This morphism
$\psi$ would satisfy $\pi \circ\psi= {\rm Id}_{\{\pm 1\}}$. By the splitting lemma for exact sequences, it follows that the sequence splits which is visibly impossible by the isomorphisms \eqref{iso}. So we proved the implication
\begin{equation}
\label{implication}
b\in B\setminus A \Rightarrow b^2 \in A\setminus A^2.
\end{equation}

Let $\epsilon $ be a fundamental unit of $\mathcal O^*_{\mathbb Q (\sqrt n)}$. Then the quotient--group $A/A^2$ exactly contains four classes, which 
are represented by the four elements
$$
1, i, \epsilon, i \epsilon.
$$
Let $b$ as above and consider the class of $b^2$ modulo $A^2$. By \eqref{implication}, we see that $b^2$ does not belong to class of $1$. Hence $b^2$ belongs to one of the three other classes. Suppose  $b^2$ belongs to the class of $i $ modulo $A^2$, then $i$ would be a square of some element of  $K_n$ and this is impossible since $\exp(2\pi i/8)$ does not belong to $K_n$.  

Consider now the case where $b^2$ belongs to the class of $\epsilon$ modulo $A^2$. Then it follows that $\epsilon = z^2$, so $\epsilon = - \alpha^2$ with $\alpha \in \Q(\sqrt{n})^\ast$ by Kummer theory. This is clearly impossible. The last case is when $b^2$ belongs to the class of $i \epsilon$ modulo $A^2$. 
This shows the implication from left to right in the claim \eqref{iff1}. 

The other implication of the claim \eqref{iff1} is clear.

We will now show that 
\begin{equation}
\label{iff2}
\textit{there exists } \epsilon \in \mathcal{O}_{\Q(\sqrt{n})}^\ast \textit{  and } z \in K_n \textit{ with } z^2 = \epsilon i \iff n \in \mathcal{E}.
\end{equation} 
Once we have shown this, Proposition \ref{Hasse}  follows immediately by combining with \eqref{iff1}. We use the convention that all $\pm$ signs appearing are independent from each other. First suppose that we have $\epsilon = a + b\sqrt{n} \in \mathcal{O}_{\Q(\sqrt{n})}^\ast$ and $z = x + y\sqrt{n} \in K_n$ with $z^2 = \epsilon i$ and $x, y \in \Q(i)$. Then we get the equations
\begin{equation}
\label{eBasis}
x^2 + ny^2 = ai, \quad 2xy = bi.
\end{equation}
Furthermore, norming the relation $z^2 = \epsilon i$ to $\Q(i)$ yields
\[
(x^2 - ny^2)^2 = -N_{\Q(\sqrt{n})/\Q}(\epsilon).
\]
First suppose $N_{\Q(\sqrt{n})/\Q}(\epsilon) = -1$, so $x^2 - ny^2 = \pm 1$. Combining this with $x^2 + ny^2 = ai$ we obtain $2ny^2 = ai \pm 1$. Now take $p$ an odd prime dividing $n$. Since $z$ is integral, we see that $2y \in \Z[i]$. Then $p \mid 2ny^2$ and hence $p \mid ai \pm 1$, which is clearly impossible. Hence $N_{\Q(\sqrt{n})/\Q}(\epsilon) = 1$ and we obtain the equality
\[
x^2 - ny^2 = \pm i.
\]
From this and equation (\ref{eBasis}) we deduce that $2x^2 = ai \pm i$ and similarly for $2ny^2$. This implies that $x = \tilde{c} \pm \tilde{c}i$ and $y = \tilde{e} \pm \tilde{e}i$ with $\tilde{c}, \tilde{e} \in \Q$. However, since $2a \in \Z$, we see that $2\tilde{c}, 2\tilde{e} \in \Z$. Plugging back in $x^2 - ny^2 = \pm i$ yields
\[
\pm 2\tilde{c}^2i \pm 2n\tilde{e}^2i = \pm i,
\]
so $2\tilde{c}^2 - 2n\tilde{e}^2 = \pm 1$. From this we see that $\tilde{c}, \tilde{e} \not \in \Z$. Writing $c = 2\tilde{c} \in \Z$ and $e = 2\tilde{e} \in \Z$ we get $c^2 - ne^2 = \pm 2$. This shows the implication from left to right in the claim \eqref{iff2}.  

For the reverse implication, suppose that $c$ and $e$ are such that $c^2 - ne^2 = \pm 2$. Then define
\[
x := \frac{1}{2}(c + ci), \quad y := \frac{1}{2}(e + ei), \quad a := \frac{c^2}{2} + \frac{ne^2}{2}, \quad b := ce.
\]
Note that $a \in \Z$, since $c$ and $e$ are both odd. Then we see that $a^2 - nb^2 = 1$, and hence $\epsilon := a + b\sqrt{n} \in \mathcal{O}_{\Q(\sqrt{n})}^\ast$. If we set $z := x + y\sqrt{n}$, we get $z^2 = \epsilon i$, and the proof of \eqref{iff2} is complete.
\end{proof}
%%%%%%%%%%%
\subsection{Distribution of the function $\omega_3 (n)$} 
It is well known that the function $\Omega (n)$ concentrates around its average value $\log \log n$, see for instance \cite [Theorem 7.20]{MoVa}. The extension of this result to the function $\omega(n)$ is easy. The case of the function $\omega_3 (n)$ is straightforward, since we have a good knowledge of the set of primes congruent to $3 \bmod 4$. We have
 
\begin{proposition} 
\label{concentration}
For every $\varepsilon >0$ there exists $c(\varepsilon ) >0$ such that, uniformly for $x \geq 2$, one has the inequalities
$$
\vert \{ n\leq x : \omega_3 ( n) \leq (1/2-\varepsilon) \log \log x\}\vert \ll x (\log x)^{-c (\varepsilon)}, 
$$
and 
$$
\vert \{ n\leq x : \omega_3 ( n) \geq (1/2+\varepsilon) \log \log x\}\vert \ll x (\log x)^{-c (\varepsilon)}, 
$$
\end{proposition}

We could even be more precise in the description of $\omega_3 (n) $ around its average value by adapting the Erd\" os--Kac Theorem, see for instance \cite[Theorem 7.21]{MoVa}.

%%%%%%%%%%%%%%%
\section{The proof of Theorem \ref{central}}
\label{Theproof}
The proof is a straightforward application of the propositions contained in \S \ref{sec2}. We start from \eqref{formdir}. Taking the $2$--parts of both sides of this equality, and taking the logarithm in basis $2$, we deduce the following equality
\begin{equation}
\label{starting}
\log_2 Q(n)-1 + \sum_{k=1} ^\infty \Bigl( \rk_{2^k}\bigl( \Cl(  n) \bigr) + \rk_{2^k}\bigl( \Cl(   {-n})\bigr)\Bigr)
= \sum_{k=1} ^\infty   \rk_{2^k}\bigl( \Cl( K_n)\bigr),
\end{equation}
which is true for any $n > 3$ odd and squarefree. If $n$ is such that $\rk_4 \bigl( \Cl( \mathbb Q (\sqrt n))\bigr) =\rk_{4}\bigl( \Cl( \mathbb Q (\sqrt {-n}))\bigr)\ =0$, then trivially, we have 
$$
\rk_{2^k} \bigl( \Cl( \mathbb Q (\sqrt n))\bigr) =\rk_{2^k}\bigl( \Cl( \mathbb Q (\sqrt {-n}))\bigr) =0 \,  ( k\geq 3 )
$$
and also thanks to Proposition \ref{8rank}, the equalities $\rk_{2^k}(\Cl (K_n) ) =0$ ($k\geq 3$). These remarks simplify \eqref{starting} into
\begin{equation}
\label{final}
\rk_4  (\Cl (K_n)) = \Bigl( \rk_{2 }\bigl( \Cl(  n)\bigr) + \rk_{2 }\bigl( \Cl(  -n)\bigr) \Bigr) -\rk_2 (\Cl (K_n)) + \log_2 Q(n) -1.
\end{equation}
Propositions \ref{Gauss4} and  \ref{rk2Kn=}   give the values of each of the two first  terms on the right--hand side of the equality
\eqref{final}  in terms of the number of prime divisors of $n$ satisfying some congruence conditions. Proposition \ref{Hasse} gives the value of $Q(n)$. Then a simple case distinction completes the proof of Theorem \ref{central}.
  
%%%%%%%%%%%%%%

\end{document}